\documentclass[11pt]{article}
\usepackage{amsmath}
\usepackage{amsfonts}
\usepackage{amsthm}
\usepackage{amssymb}
\usepackage{latexsym}
\usepackage{mathrsfs}

\def\Goo{G_{\infty}}

\def\V{\mathcal{V}}

\def\Htw{\widetilde{H}}
\def\H{\mathcal{H}}

\def\GL{\mathrm{GL}}
\def\SL{\mathrm{SL}}
\def\PGL{\mathrm{PGL}}

\def\eps{\mathrm{\epsilon}}
\def\cusp{\mathrm{cusp}}

\def\Z{\mathbf{Z}}
\def\C{\mathbf{C}}
\def\Q{\mathbf{Q}}

\def\OL{\mathcal{O}}
\def\R{\mathbf{R}}

\def\G{\mathbf{G}}
\def\Ghat{\widehat{G}_{\infty}}
\def\A{\mathbf{A}}
\def\Af{\A_{\kern-0.1em{f}}}
\def\Apf{\A^{\p}_{\kern-0.1em{f}}}
\def\Kf{K_{\kern-0.1em{f}}}
\def\p{\mathfrak{p}}
\def\q{\mathfrak{q}}
\def\La{\Lambda}
\def\Co{\mathcal{C}}
\def\calL{\mathscr{L}}
\def\Koo{K_{\infty}}

\newtheorem{theorem}{Theorem}[section]

\newtheorem{question}[theorem]{Question}
\newtheorem{conj}[theorem]{Conjecture}
\newtheorem{lemma}[theorem]{Lemma}
\newtheorem{remark}[theorem]{Remark}
\newtheorem{example}[theorem]{Example/Question}

\numberwithin{equation}{section}

\title{Bounds for Multiplicities of
Unitary Representations of Cohomological
Type in Spaces of Cusp Forms}
\author{Frank Calegari \and
Matthew Emerton}

\begin{document}
\maketitle

\section{Introduction}

Let $\Goo$ be a semisimple real Lie group
with unitary dual $\Ghat$.
The goal of this note is to produce new upper
bounds for the multiplicities with which 
representations
$\pi \in \Ghat$ of cohomological type
appear in certain spaces of cusp forms on $\Goo$.

More precisely,
we suppose that $\Goo := \G(\R\otimes_{\Q} F)$
for some connected semisimple linear algebraic group
$\G$ over a number field $F$.
Let $\Koo$ be a maximal compact subgroup of $\Goo$.
We fix an embedding $\G \hookrightarrow \GL_N$ for some $N$,
and for any ideal $\mathfrak q$ of $\OL_F$, we let
$G(\mathfrak q)$ denote the intersection of $\Goo$
with the congruence subgroup of $\GL_N(\OL_F)$ of full level
$\mathfrak q$.
We also fix an arithmetic lattice $\Gamma$ in $\Goo$
(i.e.\ a subgroup commensurable with the congruence
subgroups $G(\mathfrak q)$) 
and write $\Gamma(\q) := \Gamma \bigcap G(\q)$.
For any $\pi \in \Ghat$,
let $m(\pi,\Gamma(\q))$
denote the multiplicity with which $\pi$ occurs in
the decomposition of the regular representation of
$\Goo$ on $L^2_{\cusp}(\Gamma(\q) \backslash \Goo)$. 
Let $V(\q)$  denote the volume of the arithmetic quotient
$\Gamma(\q)\backslash \Goo$.

In terms of this notation, we may state our main results.

\begin{theorem} 
Let $\p$ be a prime ideal in $\OL_F$.
Let $\pi \in \Ghat$ be  of cohomological type. Suppose either that
$\Goo$ does not admit discrete series, or, if
 $\Goo$ admits discrete series,  that
$\pi$ contributes to cohomology in degrees other than
$\dfrac{1}{2}\dim(\Goo/\Koo)$.
Then
$$m(\pi,\Gamma(\p^k)) \ll V(\p^k)^{1 - 1/\dim(\Goo)}$$
as $k \rightarrow \infty$,
with the implied constant depending on 
$\Gamma$ \emph{and} $\p$.
\label{theorem:main}
\end{theorem}

\begin{theorem} 
Let $\p$ be a prime ideal in $\OL_F$.
Let $W$ be a finite-dimensional representation of $\Goo$,
and let $\V_{W,k}$ denote the local system on $\Gamma(\p^k) \backslash \Goo$
induced by $W$ (assuming that $k$ is taken large enough for
$\Gamma(\p^k)$ to be torsion free).
Let $n \geq 0$,
and if $\Goo$ admits discrete series, then suppose furthermore
that $n \neq \dfrac{1}{2}\dim(\Goo/\Koo)$.
Then
$$\dim H^n\bigl(\Gamma(\p^k)\backslash \Goo,\V_{W,k}\bigr)
\ll V(\p^k)^{1 - 1/\dim(\Goo)}$$
as $k \rightarrow \infty$,
with the implied constant depending on 
$\Gamma$ \emph{and} $\p$.
\label{theorem:second main}
\end{theorem}

These two theorems are evidently closely related,
in light of the results of \cite{Fr}, which show that
$H^n\bigl(\Gamma(\p^k)\backslash \Goo,\V_{W,k}\bigr)$
may be computed in terms of automorphic forms.

In the remainder of this introduction,
we discuss the relation of
Theorem~\ref{theorem:main} to prior results in this direction
before briefly describing the main ingredients in the proof
of the two theorems.

DeGeorge and Wallach~\cite{DW} established
general upper bounds for 
$m(\pi,\Gamma)$ in the case where $\Gamma$ is cocompact.
In particular (ibid, Corollary 3.2), they showed that
\begin{equation}
\label{eqn:bound}
m(\pi,\Gamma) \le
\left( \int_B |\phi(g)|^2 dg \right)^{-1}\mathrm{vol}(\Gamma\backslash \Goo),
\end{equation}
where $\phi(g) = \langle \pi(g) v, v\rangle$ is a matrix coefficient,
and  $B$ is the preimage in $\Gamma \backslash \Goo$
of a ball in $\Gamma \backslash \Goo/\Koo$ of radius equal to 
the injectivity radius of $\Gamma \backslash \Goo/\Koo$.
Suppose, however, that $\pi$ is \emph{not}
a discrete series.
In particular, 
the corresponding matrix coefficients of $\pi$ are then not 
square integrable. If  $\Gamma(\q)$
denotes the mod $\q$ congruence subgroup of $\Gamma$,
then $\mathrm{inj.rad}(\Gamma(\q)\backslash \Goo) \rightarrow
\infty$ as $N_{F/\Q}(\q) \rightarrow \infty$,
and thus, the formula of DeGeorge--Wallach implies that
\begin{equation}
\label{eqn:decay}
\lim_{N_{F/\Q}(\q) \rightarrow \infty} 
V(\q)^{-1}
\cdot m(\pi,\Gamma(\q)) =  0.
\end{equation}
For non-cocompact $\Gamma$, an analogous result
was established by Savin \cite{Savin}.

It is natural to try to improve this result
so as to obtain an estimate on the rate of decay in~(\ref{eqn:decay})
as $N_{F/\Q}(\q) \rightarrow \infty$.
If $\pi$ is non-tempered, then~(\ref{eqn:bound}) itself implies
an estimate of the form
\begin{equation}
\label{eqn:power}
m(\pi,\Gamma(\q)) \ll V(\q)^{1 - \mu}
\end{equation}
for some $\mu > 0$.   (See \cite{SX}, Lemma~1 and displayed equation~(6).)
In fact Sarnak and Xue in~\cite{SX} have conjectured
an inequality of the following form (in the case of cocompact $\Gamma$):

\begin{conj}[Sarnak--Xue]
\label{conj:sx}
For $\pi \in \Ghat$ fixed,
$$m(\pi,\Gamma(\q)) \ll  V(\q)^{(2/p(\pi)) + \eps}, \qquad \text{for all $\eps > 0$},$$
where $p(\pi)$ is the infimum over $p \ge 2$ such that the $K$-finite 
matrix coefficients of $\pi$ are in $L^p(\Goo)$.
\end{conj}

Sarnak and Xue proved their conjecture for arithmetic lattices in
$\SL_2(\R)$ and $\SL_2(\C)$, obtaining partial results
in the direction of this conjecture for $\mathrm{SU}(2,1)$.
Note, however, that their conjecture
is non trivial only for \emph{non-tempered} representations,
since for tempered representations, $p(\pi) = 2$. 
In particular,
in the tempered but non-discrete series case,
Conjecture~\ref{conj:sx} is weaker than the known result~(\ref{eqn:decay}).

In Theorem~\ref{theorem:main}, we restrict our
attention to congruence covers of the form $\Gamma(\p^k)$
for the fixed prime $\p$.
For such covers we obtain a quantitative
improvement of~(\ref{eqn:decay}) even
in the case of tempered representations
(at least for those of cohomological type;
note that non-discrete series tempered representations
of cohomological type exist precisely when $\Goo$ admits no discrete series
-- see \cite{BW}, Thm.~5.1, p.~101).
For such representations,
our result provides the first general bound of the form~(\ref{eqn:power})
for any $\mu > 0$.

\medskip

As we already noted, our two main theorems are closely related.
Indeed, Theorem~\ref{theorem:main} is an easy corollary
of Theorem~\ref{theorem:second main} (see the end of
Section~\ref{section:cohomology} below),
and most of our efforts will be concentrated on establishing
the latter result.

When studying the Betti numbers of arithmetic quotients of
symmetric spaces, it is natural to try to use tools such as
Euler characteristics and the Lefschetz trace formula.
When applied to analyzing
contributions from the discrete series,
such methods tend to be very powerful; for example,
the $(\mathfrak g,\mathfrak k)$-cohomology of a discrete series
representation is concentrated in a single dimension \cite{BW},
and so no cancellations occur when taking alternating sums.
However, in other situations, these methods can be 
useless. 
For example, if $\pi$ is tempered but not discrete series,
then the Euler characteristic of its $(\mathfrak g,\mathfrak k)$-cohomology
vanishes \cite{BW}.  Similarly,
in situations where 
the symmetric space is a real manifold of odd dimension $n$, Poincare
duality leads to cancellations in  the natural
sum $\displaystyle{\sum_{k=0}^n (-1)^k \dim(H^k)}$.
One is thus forced to find different techniques.
The proof of Theorem~\ref{theorem:second main}
takes as input the inequality~(\ref{eqn:decay}) of~\cite{DW} and~\cite{Savin}
and a spectral sequence from~\cite{Emerton}, proceeding via a bootstrapping
argument relying on non-commutative Iwasawa theory.

\medskip

{\em Acknowledgments.} The second author would like to thank
Peter Sarnak for a very stimulating conversation on the subject
of this note.

\section{Iwasawa Theory}

Let $G \subseteq \GL_N(\Z_p)$ be an
analytic pro-$p$ group.
Let $G_k = G \cap (1 + p^k M_N(\Z_p))
 \subseteq \GL_N(\Z_p)$.
The subgroups $G_k$ form a fundamental set of open neighbourhoods
of the identity in $G$, and moreover, for
large $k$, there exists a constant $c$ such
that $[G:G_k] = c \cdot p^{d k}$, where
$d = \dim(G)$.

Fix a finite extension $E$ of $\Q_p$ with ring of integers $\OL_E$.
Write $\La = \OL_E[[G]]$
and $\La_E = E\otimes_{\OL_E} \La$.
The module theory of
$\La$ falls under the rubric of Iwasawa theory.
A fundamental result of Lazard~\cite{L} states
that $\La$ is Noetherian; the same is thus true
of the ring $\La_E$.  The rings $\La$ and $\La_E$ are  non-commutative
domains admitting a common field of
fractions which we will denote by $\calL$. Thus, $\calL$ is a division
ring which contains $\La$ and $\La_E$ and is flat over each of them
(on both sides).
If $M$ is a finitely generated left $\La$-module (resp.\ $\La_E$-module),
then $\calL \otimes_{\La} M$ (resp.\ $\calL \otimes_{\La_E} M$)
is a finite-dimensional left
$\calL$-vector space; we define the {\em rank} of $M$
to be the $\calL$-dimension of this vector space.
Note that 
rank is additive in short exact sequences of finitely generated
$\La$-modules (resp.\ $\La_E$-modules),
by virtue of the flatness of $\calL$ over $\La$ and $\La_E$.

Recall that a continuous representation of $G$ on an
$E$-Banach space $V$ 
is called {\em admissible} if
its topological $E$-dual $V'$ (which is naturally a $\La_E$-module)
is finitely generated over $\La_E$.  (See \cite{ST}; a key point is
that since $\La_E$ is Noetherian, the category of admissible continuous
$G$-representations is abelian. Indeed, passing to topological duals
yields an anti-equivalence with the abelian category of finitely generated
$\La_E$-modules.)
We define the {\em corank} of an admissible $G$-representation
to be the rank of the finitely generated $\La_E$-module $V'$.

A coadmissible $G$-representation $V$ is not determined by the collection
of subspaces of invariants $V^{G_k}$ ($r \geq 1$).
However, the following result 
(Theorem 1.10 of Harris~\cite{Harris}) shows that its corank is so determined.

\begin{theorem}[Harris]
Let $V$ be an $E$-Banach space equipped with
an admissible continuous $G$-representation
and let $d = \dim(G)$.
Then as $k \rightarrow \infty$,
$$\dim_E V^{G_k} = 
r \cdot [G:G_k] +  O(p^{(d-1)k}) = 
 r \cdot c \cdot p^{d k} +  O(p^{(d-1)k}),$$
where $r$ is the corank of $V$ and
$c$ depends only on $G$.
\label{theorem:harris}
\end{theorem}

Using this result, we may obtain bounds on the dimensions of 
the continuous cohomology groups
$H^i(G_k,V)$ in terms of $k$ for admissible continuous $G$-representations
$V$.  (Let us remark that the continuous $G_k$-cohomology on the category
of admissible continuous $G$-representations may also be computed
as the right derived functors of the functor of $G_k$-invariants; see
Prop.~1.1.3 of \cite{Emerton}.)
\begin{lemma} Let $V$ be an admissible continuous $G$-representation. 
For each $i \ge 1,$ 
$$\dim_E H^i(G_k,V) \ll p^{(d-1)k},$$
as $k \rightarrow \infty$.
\label{lemma:small}
\end{lemma}
\begin{proof}
Let $\Co:=\Co(G,E)$ denote the Banach space of continuous $E$-valued
functions on $G$, equipped with the right regular $G$-action.
The module $\Co$ has corank one (indeed, it is cofree -- its
topological dual is free of rank one over $\La_E$).
Moreover, $\Co$ is injective in the abelian
category of admissible $G$-representations
and is therefore acyclic. 
If $V$ is an admissible continuous $G$-representation,
then there exists an exact sequence
$$
0 \longrightarrow V \longrightarrow \Co^n \longrightarrow W \longrightarrow 0$$
of admissible continuous $G$-representations
for some integer $n \geq 0$.  
Since $\Co$ is acyclic, from the long
exact sequence of cohomology we obtain the following:
\begin{equation}
\label{eqn:ses}
0 \longrightarrow V^{G_k} \longrightarrow  (\Co^{G_k})^{n} \longrightarrow
{W}^{G_k} \longrightarrow
 H^1(G_k,V) \longrightarrow 0,
\end{equation}
\begin{equation}
\label{eqn:dimension shift}
H^{i}(G_k,V) \simeq H^{i-1}(G_k,W), \ i \ge 2.
\end{equation}
The lemma for $i = 1$ follows
from a consideration of~(\ref{eqn:ses}),
taking into account Theorem~\ref{theorem:harris}
and the fact that
corank of $W$ is
equal to $n$ minus the corank of $V$
(since corank is additive in short exact sequences). 
We now proceed by induction on $i$. Assume the result for $i \le m$
and all admissible continuous representations, in particular for $W$.
The result for $i = m+1$ then follows directly from the
isomorphism~(\ref{eqn:dimension shift}).
This completes the proof.
\end{proof}

\section{Cohomology of Arithmetic Quotients of Symmetric
Spaces}
\label{section:cohomology}

We now return to the situation
considered in the introduction
and use the notation introduced there.
In particular, we fix a connected
semisimple linear group $\G$ over $F$,
an embedding $\G \hookrightarrow \GL_N$ over $F$,
an arithmetic lattice $\Gamma$ of the associated
real group $\Goo$, and a prime $\p$ of $F$.

If we write $\displaystyle G := \lim_{\overleftarrow{k}}
 \Gamma/\Gamma(\p^k),$
then $G$ is a compact open subgroup of the $p$-adic Lie group
$\G(F_{\p})$ (where $F_{\p}$ denotes the completion of $F$ at $\p$);
alternatively, we may define $G$ to be the closure of $\Gamma$ in 
$\G(F_{\p})$.
If we replace $\Gamma$ by $\Gamma(\p^{k})$ for some sufficiently
large value of $k$ (i.e. discarding 
finitely many
initial terms in the descending sequence of lattices $\Gamma(\p^k)$),
then $G$ will be pro-$p$ and hence, will be an analytic pro-$p$-group.
Note that $\Gamma$ is a dense subgroup of $G$. 
Let $e$ and $f$ denote respectively
the ramification and inertial indices of $\p$ in $F$
(so that $[F_{\p}:\Q_p] = e\kern-0.1em{f}$).
For each $k\geq 0$,
write $G_k$ to denote the closure of $\Gamma(\p^{e k})$ in $G$.
Alternatively,
if we consider the embedding
$$\G(F_{\p}) \hookrightarrow \GL_N(F_{\p})\hookrightarrow \GL_{e\kern-0.1em{f} \kern-0.1em{N}}(\Q_p),$$
then $G_k = G \cap (1 + p^k M_{e\kern-0.1em{f} \kern-0.1em{N}}(\Z_p))$;
thus, our notation is compatible with that of the preceding section.
We let $d$ denote the dimension of $G$; note that
$d = \frac{e\kern-0.1em{f}}{[F:\Q]} \cdot \dim(\Goo).$

For each $k \geq 0,$ we write
$$Y_k := \Gamma(\p^{e k})\backslash \Goo/\Koo.$$
There is a natural action of $G$ on $Y_k$
through its quotient
$G/G_k$ ($\cong \Gamma/\Gamma(\p^{e k})$),
which is compatible with the projections
$Y_k \rightarrow Y_{k'}$
for $0 \leq k' \leq k$.

\medskip

Fix a finite-dimensional representation $W$ of $\G$ over $E$,
and let $W_0$ denote a $G$-invariant $\OL_E$-lattice in $W$.
Let $\V_k$ denote the local system of
free finite rank $\OL_E$-modules on
$Y_k$ associated to $W_0$, and denote by $\V_k$ the pull-back
of $\V_0$ to $Y_k$ for any $k \ge 0$.
If $0 \leq k' \leq k,$ then the sheaf
$\V_k$ on $Y_k$ is
naturally isomorphic to the pull-back of the sheaf
$\V_{k'}$ on $Y_{k'}$ 
under the projection
$Y_k \rightarrow Y_{k'}.$
In particular the 
sheaf $\V_k$ is $G/G_k$-equivariant.

Recall the following definitions from from
\cite{Emerton}, p.~21:
$$
\Htw^n(\V):=  
\lim_{\overleftarrow{s}}  \lim_{\overrightarrow{k}}
 H^j(Y_k,\V_k/p^s),
\qquad
\Htw^n(\V)_E := 
E \otimes_{\OL_E} \Htw^n(\V).
$$
Each
$\Htw^n(\V)$ is a 
$p$-adically complete $\OL_E$-module, equipped with a left $G$-action
in a natural way,
and hence, each
$\Htw^n(\V)_E$ has a natural structure
of $E$-Banach space and is equipped with a continuous 
left $G$-action.  In fact, they are admissible continuous
representations of $G$ (\cite{Emerton}, Thm.~2.1.5~(i)), and
in particular,
Theorem~\ref{theorem:harris} and Lemma~\ref{lemma:small}
apply to them. 
(Note that the results of \cite{Emerton} are stated in the
ad\`elic language.  We leave it to the reader to make
the easy translation to the more classical language
we are using in this paper.)

The following result, which is
Theorem~2.1.5~(ii) of~\cite{Emerton}, p.~22,
is a ``control theorem'' relating $G_k$ invariants in
$\Htw^i(\V)^{G_k}_E$ to the classical cohomology classes
$H^j(Y_k,\V_k) \otimes E$. 

\begin{theorem}
\label{thm:ss}
Fix an integer $k$.
There is a spectral sequence
$$E^{i,j}_2(Y_k) = H^i(G_k,\Htw^j(\V)_E)
\Longrightarrow H^{i+j}(Y_k,\V_k)_E.$$
\end{theorem}

One should view this spectral sequence as a version of
the Hochschild-Serre spectral sequence ``compatible in the
$G$-tower.''

\medskip

\begin{theorem}
\label{thm:rank}
For any $n \geq 0,$
if $r_n$ denotes the corank of $\Htw^n(\V)$,
then
$$\dim_E H^n(Y_k,\V_k)_E =
r_n\cdot c \cdot p^{d k} + O(p^{(d-1) k})$$
as $k \rightarrow \infty$. 
(Here $c$ denotes the constant appearing in the statement
of Theorem~\ref{theorem:harris}; it depends only on $G$.)
\end{theorem}

\begin{proof}
For each $i,j \geq 0$ and $l \geq 2$, 
let
$E^{i,j}_{l}(Y_k)$
denote the terms in the
spectral sequence of Theorem~\ref{thm:ss}. 
Since $\Htw^j(\V)$ is admissible,
Lemma~\ref{lemma:small} implies that
$\dim_E H^i(G_k,\Htw^j(\V)_E) \ll p^{(d-1)k}$,
and thus, 
$\dim_E E^{i,j}_{l}(Y_k) \ll p^{(d-1)k},$
as $k \rightarrow \infty$
(since $E^{i,j}_{l}(Y_k)$
is a subquotient of $E^{i,j}_2(Y_k) := 
H^i(G_k,\Htw^j(\V)_E)$ 
for ${l} \ge 2$).
Theorem~\ref{theorem:harris} shows that
$$\dim_E E^{0,n}_2(Y_k) = 
\dim_E \Htw^n(\V)^{G_k}_E
= r_n \cdot c \cdot p^{d k} + O(p^{(d-1) k}).$$
On the other hand, since the spectral sequence of Theorem~\ref{thm:ss}
is an upper right quadrant exact sequence,
$E^{0,n}_{\infty}$
is obtained by taking finitely many successive kernels
of differentials $d_{l}$ to $E^{i+{l},j-{l}+1}_{l}$, which all
have order $\ll p^{(d-1)k}$ by the first part of our
argument. Thus,
$$\dim_E E^{0,n}_{\infty} = r_n \cdot c \cdot p^{d k} +  O(p^{(d-1) k}).$$
Since
$H^n(Y_k,\V_k)_E$ admits a finite length filtration whose
associated graded pieces are isomorphic to
$E^{i,j}_{\infty}$ for $i + j = n$,
we conclude that
$\dim_E H^n(Y_k,\V_k)_E 
= r_n \cdot c \cdot p^{d k} + O(p^{(d-1) k}),$
as claimed.
\end{proof}

The following lemma quantifies the precise relationship between
multiplicities and the dimensions of cohomology groups that we will
require to deduce Theorem~\ref{theorem:main} from
Theorem~\ref{theorem:second main}.

\begin{lemma}
\label{lem:passage}
Fix a cohomological degree $n$,
and let $S$ denote the set of isomorphism classes $[\pi]$ of 
$\pi \in \Ghat$ that contribute to cohomology with
coefficients in $\mathcal V$ in degree $n$.
Then 
$$\sum_{[\pi] \in S} m(\pi,G_k) 
\asymp
\dim_E H^n_{\cusp}(Y_k,\V_k)_E.$$
\end{lemma}
\begin{proof}
Since the set $S$ is finite, there is an integer
$d \geq 1$ so that 
$$1 \leq \dim H^n(\mathfrak{g},\mathfrak{k}; \pi \otimes W) \leq d$$
for each isomorphism class $[\pi] \in S$.  This implies that
$$\sum_{[\pi] \in S} m(\pi,G_k) 
\leq
\dim_E H^n_{\cusp}(Y_k,\V_k)_E
\leq
d\bigl(\sum_{[\pi] \in S} m(\pi,G_k) \bigr)$$
for each $k \geq 0$.
\end{proof}

We can now prove our main result.

\begin{theorem}
\label{theorem:third main}
Let $n \geq 0,$
and suppose that either 
$\Goo$ does not admit discrete series
or else that $n \neq
\dfrac{1}{2}\dim(\Goo/\Koo).$
Then 
$$\dim_E H^n(Y_k,\V_k)_E \ll p^{(d-1)k}$$
as $k \rightarrow \infty$, for all $n \geq 0$.
\end{theorem}
\begin{proof}
In the case when $\Goo$ admits discrete series,
recall that these contribute to cohomology only in the dimension
$\dfrac{1}{2}\dim(\Goo/\Koo)$
(\cite{BW}, Thm.~5.1, p.~101).
Thus, under the assumptions of the theorem,
there is no contribution from the discrete series to $H^n_{\cusp}(Y_k,\V_k)$.
The inequality~(\ref{eqn:decay}) of~\cite{DW}
and~\cite{Savin}, together with Lemma~\ref{lem:passage}
and the main result of~\cite{RS}
(which states that
$\dim_E \bigl( H^n(Y_k,\V_k)/H^n_{\cusp}(Y_k,\V_k)\bigr) = o(p^{d k})$),
thus shows that
$\dim_E H^n(Y_k,\V_k)_E = o(p^{d k})$
as $k \rightarrow \infty$, for all $n \geq 0$.
From Theorem~\ref{thm:rank}, we then infer that
each $\Htw^n(\V)$ has corank $0$.  Another application
of the same theorem now gives our result.
\end{proof}

Note that $V(\p^{e k}) \sim [G:G_k] \sim c \cdot p^{d k}$;
thus Theorem~\ref{theorem:third main}
implies Theorem~\ref{theorem:second main} since
\begin{equation}
\label{eqn:dimensions}
\dim(G) \leq \dim(\Goo). 
\end{equation}
Theorem~\ref{theorem:second main}
and Lemma~\ref{lem:passage}
together imply Theorem~\ref{theorem:main}.

\begin{remark}
{\em
We have equality in~(\ref{eqn:dimensions})
precisely when $\p$ is the unique prime lying over $p$ 
in $\OL_F$.   If there is more than one prime lying over $p$,
then $\dim(G)$ is strictly less than $\dim(\Goo)$,
and we obtain a corresponding improvement in the bounds of
Theorems~\ref{theorem:main} and~\ref{theorem:second main}, namely
(in the notation of their statements), that
$$m(\pi,\Gamma(\p^k)) \quad \text{ and } \quad
\dim_E H^i(Y(\p^k),\V_{W,k})_E \,\, \ll \, V(\p^n)^{1 - 1/\dim(G)}$$
(where, as we noted above,
$\dim(G) = \frac{e\kern-0.1em{f}}{[F:\Q]} \cdot \dim(\Goo),$
with $e$ and $f$ being the ramification and inertial index of $\p$
respectively).
}
\end{remark}

\begin{example}
{\em
Let $F/\Q$ be an imaginary quadratic field, and let
$\G = \SL_{2/F}$.
The corresponding symmetric
space $\Goo/\Koo = \SL_2(\C)/\mathrm{SU}(2)$
is a real hyperbolic three space $\H$,
and the quotients $Y$ are commensurable with
the Bianchi manifolds $\H/\PGL_2(\OL_K)$.
Choose a local system $\V_0$ associated to some
finite-dimensional representation $W$ of $\Goo = \GL_2(\C)$
and a congruence subgroup $\Gamma$.
Assume that $p = \p \overline{\p}$ splits in $\OL_F$,
and apply Theorem~\ref{theorem:third main} to the $\p$-power tower.
We obtain the inequality
$$H^1_{\cusp}(Y_k,\V_k) \ll p^{2k}$$
as $k \rightarrow \infty$.
It is natural to ask how tight this inequality is.

The main result of Calegari--Dunfield~\cite{CD} 
shows 
that there exists at least one $(F,\Gamma,\p)$ for which 
$$H^1_{\cusp}(Y_k,\C) =0 $$ for all $k$.
On the other hand,
if there exists at least one newform on $\Gamma(\p^k)$ for some $k$,
then a consideration of the associated oldforms shows that 
$$H^1_{\cusp}(Y_k,\V_k) \gg p^{k}$$
as $k \rightarrow \infty$.
Are there situations in which this lower bound gives the true rate of growth?
}
\end{example}

\begin{remark}
{\em 
Our results are most interesting
in the case when $\Goo$ does not admit any discrete series,
since, as we noted in the introduction, in this case (and only
in this case),
$\Goo$ admits (non-discrete series) tempered representations
of cohomological type.

On the other hand, Theorem~\ref{thm:rank} does have a
consequence in the case when $\Goo$ admits discrete series
which may be of some interest.  Recall the following
result from \cite{Savin} (established in \cite{DW}
in the cocompact case):  if $\pi \in \Ghat$ lies in the
discrete series, then
\begin{equation}
\label{eqn:ds}
m(\Gamma(\pi^k),\pi) = d(\pi) V(\p^k) + o(V(\p^k))
\end{equation}
as $k \rightarrow \infty$.
Fix a finite-dimensional representation $W$ of $\Goo,$
and let $\Ghat(W)_d$ denote the subset of $\Ghat$ consisting
of discrete series representations that contribute to
cohomology with coefficients in $W$.
Summing over all $\pi \in \Ghat(W)_d$, we obtain the formula
\begin{equation}
\label{eqn:ds sum}
\dfrac{1}{|\Ghat(W)_d|} \sum_{\pi \in \Ghat(W)_d} m(\Gamma(\pi^k),\pi) =
 d(\pi) V(\p^k) + o(V(\p^k)).
\end{equation}
(A result first proved in \cite{RS}.)
The following result provides an improvement in the error
term of~(\ref{eqn:ds sum}).
}
\end{remark}

\begin{theorem}
\label{thm:ds}
There exists $\mu > 0$ such that
$$\dfrac{1}{|\Ghat(W)_d|} \sum_{\pi \in \Ghat(W)_d} m(\Gamma(\pi^k),\pi) =
d(\pi) V(\p^k) + O(V(\p^k)^{1-\mu}).$$
\end{theorem}

\begin{proof}
Let $n = \dfrac{1}{2}\dim(\Goo/\Koo).$
As already noted, it follows from \cite{BW}, (Thm.~5.1, p.~101),
that all non-discrete series contributions to $H^n_{\cusp}(Y_k,\V_k)$
are non-tempered.  The same result shows that each discrete series
has one-dimensional $(\mathfrak g, \mathfrak k)$-cohomology in dimension $n$.
As we recalled in the introduction,
the multiplicity of any non-tempered representations
is bounded by $V(\p^k)^{1 - \mu}$ for some $\mu > 0$ \cite{SX},
and thus, Theorem~\ref{thm:rank} and (the proof of) Lemma~\ref{lem:passage}
show that
$$\dfrac{1}{|\Ghat(W)_d|} \sum_{\pi \in \Ghat(W)_d} m(\Gamma(\pi^k),\pi) =
C \cdot V(\p^k) + O(V(\p^k)^{1-\mu}).$$
Comparing this formula with~(\ref{eqn:ds sum}) yields the theorem.
\end{proof}

\begin{question}
{\em 
Does the result of Theorem~\ref{thm:ds} hold term-by-term?
That is, does~(\ref{eqn:ds}) admit an improvement of the form
$$
m(\Gamma(\pi^k),\pi) \ {\buildrel ? \over =} \ 
d(\pi) V(\p^k) + O(V(\p^k)^{1-\mu})$$
for some $\mu > 0$?
}
\end{question}


\begin{thebibliography}{99}

\small



\bibitem{BW}
Borel, A.; Wallach, N.
Continuous cohomology, discrete subgroups,
and representations of reductive groups.
Annals of Mathematics Studies, 94. Princeton University Press, Princeton, N.J.; University of Tokyo Press, Tokyo, 1980. xvii+388 pp.

\bibitem{CD}
Calegari, F.; Dunfield, N.
Automorphic forms and rational homology 3-spheres.
Geom. Topol. 10 (2006), 295--329.


\bibitem{DW}
DeGeorge, D.; Wallach, N.
Limit formulas for multiplicities in $L\sp{2}(\Gamma \backslash G)$.
Ann. Math. (2) 107 (1978), no. 1, 133--150. 

\bibitem{Emerton}
Emerton, M.
On the interpolation of systems of eigenvalues attached to automorphic Hecke eigenforms.
Invent. Math. 164 (2006), no. 1, 1--84. 

\bibitem{Fr}
Franke, J.
Harmonic Analysis in Weighted $L_2$-spaces.
 Ann. Sci. École Norm. Sup. (4)  31  (1998),  no. 2, 181--279.


\bibitem{Harris}
Harris, M.
Correction to: ``$p$-adic representations arising from descent on abelian varieties''
 [Compositio Math. 39 (1979), no. 2, 177--245].
Compositio Math. 121 (2000), no. 1, 105--108. 

\bibitem{L}
Lazard, M.
Groupes analytiques $\p$-adiques,
Publ.~Math.~IHES 26 (1965).

\bibitem{RS}
Rohlfs, J.; Speh, B.
On limit multiplicities of representations with cohomology in the cuspidal spectrum.
Duke Math. J. 55 (1987), no. 1, 199--211.

\bibitem{SX}
Sarnak, P.; Xue, X.
Bounds for multiplicities of automorphic representations.
Duke Math. J. 64 (1991), no. 1, 207--227. 

\bibitem{Savin}
Savin, G.
Limit multiplicities of cusp forms.
Invent. Math. 95 (1989), no. 1, 149--159. 

\bibitem{ST}
Schneider, P.; Teitelbaum, J.
Banach space representations and Iwasawa theory,
Israel.~J.~Math.~127 (2002), 359--380.

\end{thebibliography}
\end{document}